\documentclass[12pt,a4paper]{article}
\usepackage{indentfirst}
\usepackage{amsmath,amssymb,amsfonts,amsthm,graphics}
\usepackage{makeidx}
\usepackage{amsmath,amssymb,amsfonts,amsthm,graphics,graphicx}
\usepackage{makeidx}
\usepackage{color,pict2e}
\usepackage{xcolor}
\usepackage{amsmath}
\usepackage{cite}
\usepackage{mathrsfs}
\usepackage{tikz}

\newtheorem{theorem}{Theorem}[section]

\newtheorem{lem}[theorem]{Lemma}
\newtheorem{conj}[theorem]{Conjecture}

\theoremstyle{definition}

\def\-{\mbox{--}}

\begin{document}
\title{\large\bf Path decompositions of
Eulerian graphs}

\author{Yanan Chu\thanks{School of Mathematical Sciences, Suzhou University of Science and Technology, Jiangsu 215009, China. This research is supported by National Natural Science Foundation of China under grant No. 12201447 and 12271099.},
   Yan Wang\thanks{Corresponding author. E-mail: yan.w@sjtu.edu.cn. School of Mathematical Sciences, Shanghai Jiao Tong University, Shanghai 200240, China. This research is supported by National Key R\&D Program of China under grant No. 2022YFA1006400, National Natural Science Foundation of China under grant No. 12571376 and 12201400, and Shanghai Municipal Education Commission under grant No. 2024AIYB003.}}

\date{}

\maketitle
\begin{abstract}
Gallai's conjecture asserts that every connected graph on $n$ vertices can be decomposed into $\frac{n+1}{2}$ paths. For general graphs (possibly disconnected), it was proved that every graph on $n$ vertices can be decomposed into $\frac{2n}{3}$ paths. This is also best possible (consider the graphs consisting of vertex-disjoint triangles). Lov\'{a}sz showed that every $n$-vertex graph with at most one vertex of even degree can be decomposed into $\frac{n}{2}$ paths. However, Gallai's conjecture is difficult for graphs with many vertices of even degrees. Favaron and Kouider verified Gallai's conjecture for all Eulerian graphs with maximum degree at most $4$. In this paper, we show if $G$ is an Eulerian graph on $n \ge 4$ vertices and the distance between any two triangles in $G$ is at least $3$, then $G$ can be decomposed into at most $\frac{3n}{5}$ paths.

\end{abstract}

\section{Introduction}
\noindent
All graphs considered in this paper are finite, undirected and simple. A \emph{path decomposition} of a graph $G$ is a set of edge-disjoint paths whose union contains all the edges of $G$. Let $p(G)$
denote the minimum number of paths needed in a path decomposition of
$G$.
Gallai proposed the following well-known conjecture on path decomposition (see \cite{L}).

\begin{conj}\label{conj-gallai}
If $G$ is a connected graph on n vertices,
then $p(G)\leq \frac{n+1}{2}$.
\end{conj}

Conjecture \ref{conj-gallai} has been studied extensively. For example, the following contribution is due to Lov\'{a}sz \cite{L}.

\begin{theorem}[Lov\'{a}sz \cite{L}] \label{thm-Lov}
Let $G$ be a graph (possibly disconnected) on $n$ vertices. If $G$ contains at most one vertex of even degree, then $p(G)\leq  \frac{n}{2}$.
\end{theorem}
Note that the bound $\frac{n+1}{2}$ of Conjecture \ref{conj-gallai} is sharp for connected graphs (consider the complete graphs of odd order).
For disconnected graphs, Donald\cite{Donald} extended Theorem \ref{thm-Lov} by showing that $p(G)\leq \frac{3n}{4}$ for any graph $G$ on $n$ vertices.
Furthermore, Dean and Kouider\cite{DK}, and Yan\cite{Yan} independently showed that $\frac{2n}{3}$ is best possible since the graph consisting of $k$ vertex-disjoint triangles needs at least $2k$ paths in any path decomposition. The \emph{girth} of a graph $G$ is the minimum length of a cycle in $G$. Harding and McGuinness \cite{HM} considered graphs with girth $g\geq 4$ and showed that $\frac{(g+1)n}{2g}$ is sufficient. Chu, Fan and Zhou \cite{CFZ} proved that $p(G) \le \frac{3n}{5}$ for triangle-free graphs.
For sufficiently large graph $G$ with linear minimum degree, Gir\~{a}o, Granet, K\"{u}hn and Osthus \cite{GGKO} showed that $p(G)\leq \frac{n}{2}+o(n)$.

As a consequence of Theorem \ref{thm-Lov}, Conjecture \ref{conj-gallai} holds for graphs with all vertices of odd degrees.
The main difficulty in studying Conjecture \ref{conj-gallai} is to handle graphs with many vertices of even degrees. Favaron and Kouider \cite{FK} verified Conjecture \ref{conj-gallai} for all Eulerian graphs with maximum degree at most $4$. For $2k$-regular graphs ($k\geq 3$), Botler and Jim\'{e}nez \cite{BJ} showed that if $G$ has girth at least $2k-2$ and a pair of disjoint perfect matchings, then $p(G)\leq \frac{n}{2}$.
The \textit{$E$-subgraph} of a graph $G$ is the subgraph induced by
the vertices of even degree in $G$.  Pyber \cite{P} extended Theorem \ref{thm-Lov} by showing that $p(G)\leq \frac{n}{2}$ if the $E$-subgraph of $G$ is a forest. A forest can be regarded as a graph in which each block is a single edge or an isolated vertex. Fan \cite{Fan} generalized Pyber's result by proving that if each block of the $E$-subgraph of $G$ is a triangle-free graph with maximum degree at most 3, then $p(G)\leq \frac{n}{2}$. Botler and Sambinelli \cite{BS} further extended this by allowing the components of
the $E$-subgraph to contain any number of blocks with triangles as long as they are subgraphs of a family of special graphs.
Other progress on Conjecture \ref{conj-gallai} can be found in \cite{BP,BJS,BSC,FHZ,ZLH}.

The \emph{distance} between two triangles $T_1$ and $T_2$ in a graph is the length of the shortest path between any vertex of $T_1$ and any vertex of $T_2$.
In this paper, we prove the following.

\begin{theorem}\label{main-thm}
Let $G$ be an Eulerian graph on $n \ge 4$ vertices. If the distance between any two triangles in $G$ is at least $3$, then $p(G)\leq  \frac{3n}{5}$.
\end{theorem}

Note that any path decomposition of a triangle requires at least two paths. So we need $n \ge 4$ in the theorem above.

A \emph{triangle removal set} of a graph $G$ is a set of vertices $R=\{v_1,\ldots,v_k\}$ such that $G-R$ is triangle-free and $R$ has the following properties: for any distinct $i, j\in [k]$, (i) $d_G(v_i)\geq 4$; (ii) all the vertices in $V(G)\backslash R$ have even degree in $G$; and (iii) ($N_G(v_i)\cup \{v_i\})\cap (N_G(v_j)\cup \{v_j\})=\emptyset$.

In fact, we show the following theorem which is slightly stronger than Theorem 1.3.

\begin{theorem}\label{second-thm}
Let $G$ be a graph on $n$ vertices with no component isomorphic to a triangle.
If $G$ has a triangle removal set, then $p(G)\leq  \frac{3n}{5}$.
\end{theorem}

Note that Theorem \ref{main-thm} is related to the cycle decompositions of Eulerian graphs.
In 1960s, Erd\H{o}s and Gallai conjectured that every $n$-vertex Eulerian graph can be decomposed into $O(n)$ cycles (see \cite{EGP}).
An equivalent and more well-known form is that every $n$-vertex graph can be decomposed into $O(n)$ cycles and edges.
This conjecture was established for typical random graphs in \cite{CFS,GKO,KKS}.
Conlon, Fox and Sudakov \cite{CFS} proved the conjecture for graphs with linear minimum degree while the asymptotically sharp bound of $(\frac{3}{2} + o(1))n$ was shown in \cite{GGKO}.
For general graphs, Erd\H{o}s and Gallai proved that every $n$-vertex graph can be decomposed into $O(n \log n)$ cycles and edges by iteratively removing a longest cycle.
Conlon, Fox and Sudakov \cite{CFS} obtained a major breakthrough by showing that $O(n \log \log n)$ cycles and edges suffice.
Recently, Bucić and Montgomery \cite{BM} further improved the bound to $O(n \log^\star n)$ where $\log^\star$ is the iterated logarithm function.

The organization of this paper is as follows. In the next section, we show several useful lemmas. We prove Theorems \ref{main-thm} and \ref{second-thm} in Section 3.

We conclude this section with some useful notations.
Let $G$ be a graph.
The set of vertices and edges of $G$ are denoted by $V(G)$ and $E(G)$, respectively.
Let $u, v\in V(G)$. The edge with ends $u$ and $v$ is denoted by $uv$.
We use $N_{G}(v)$ and $d_{G}(v)$ to denote the set and the number
of neighbors of $v$ in $G$, respectively.
For $S\subseteq V(G)$, $G-S$ denotes the induced subgraph of $G$ on $V(G) \setminus S$.
Let $F$ be a set of edges (resp. non-edges) of $G$.
We denote the graph obtained from $G$ by deleting (resp. adding) all the edges of $F$ by $G-F$ (resp. $G+F$).
When $F=\{uv\}$, we simply write $G-uv$ (resp. $G+uv$) instead of $G-\{uv\}$ (resp. $G+\{uv\}$).
Let $P$ be a path and $x$, $y$ be two vertices on $P$. The subpath of $P$ with end vertices $x$ and $y$ is denoted by $xPy$.
Note if $x=y$, then $xPy$ is just the vertex $x$.
We say that a graph $G$ can be \emph{decomposed into} $k$ paths if $G$ has a path decomposition $\mathcal{P}$ such that $|\mathcal{P}|=k$.
Given a path decomposition $\mathcal{P}$ of $G$ and a vertex $v\in V(G)$, we denote by $\mathcal{P}(v)$ the number of paths in $\mathcal{P}$ with $v$ as an end vertex. Note that if $v$ is a vertex of $G$ with odd degree, then $\mathcal P(v)\geq1$.

\section{Decomposition lemmas}

\noindent
We need the following lemma in \cite{CFL}.

\begin{lem}[Lemma 2.1 in \cite{CFL}]\label{lem-K_5^-}
Let $G$ be a connected graph consisting of edge-disjoint path $P$ and cycle $C$. If $|V(P)\cap V(C)|\leq 5$, then $G$ can be decomposed into two paths, unless $G$ is isomorphic to the exceptional graph (see Figure $1$).
\end{lem}
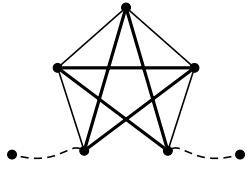
\begin{figure}[htbp]
\begin{center}
\begin{tikzpicture}[scale=0.5]

 \node  at(0,-0.2) [circle, draw=black, fill=black, scale=0.3pt] (x1) {};

 \node at (-1.1,-4)[circle, draw=black, fill=black, scale=0.3pt] (v1)  {};

 \node  at (1.8,-1.8)  [circle, draw=black, fill=black, scale=0.3pt] (v2)  {};

 \node at (-1.8,-1.8)   [circle, draw=black, fill=black, scale=0.3pt] (v3) {};

 \node at (1.1,-4)   [circle, draw=black, fill=black, scale=0.3pt]  (v4) {};
 \node at (-3,-4.1)   [circle, draw=black, fill=black, scale=0.3pt]  (y1) {};
 \node at (3,-4.1)   [circle, draw=black, fill=black, scale=0.3pt]  (y2) {};

 \draw [semithick](v1)--(v3)--(x1)--(v2)--(v4);
 \draw [very thick](v1)--(v2)--(v3)--(v4)--(x1)--(v1);
 \draw [semithick,densely dashed](v1).. controls (-1.5,-3.8) and (-1.8,-4.4).. (y1);
 \draw [semithick,densely dashed](v4) .. controls (1.5,-3.8) and (1.8,-4.4) .. (y2);

\end{tikzpicture}
\end{center}\caption{\small The exceptional graph with $|V(C)|=5$ and $|V(P)\cap V(C)|=5$ ($C$ is depicted by thick solid lines; $P$ is depicted by thin lines, with dashed lines used to represent its subpaths).} \label{fig1}
\end{figure}

We obtain the following lemma, which is an extension of Lemma \ref{lem-K_5^-}. Note that the condition $|V(P)|\leq 4$ in Lemma \ref{lem-at most 5 cycles to paths} is necessary because of the
exceptional graph in Figure 1 (note $|V(P)|=5$ and $|\mathcal{C}|=1$ in the exceptional graph). Let $\mathcal{C}$ be a set of edge-disjoint subgraphs of $G$. Write $V(\mathcal{C})=\bigcup_{C\in \mathcal{C}}V(C)$ and $E(\mathcal{C})=\bigcup_{C\in \mathcal{C}}E(C)$.

\begin{lem}\label{lem-at most 5 cycles to paths}
Let $\mathcal{C}$ be a set of edge-disjoint cycles and $P$ be a path such that $E(P)\cap E(C)=\emptyset$ and $V(P)\cap V(C)\neq \emptyset$ for every cycle $C\in \mathcal{C}$. If $|V(P)|+ |\mathcal{C}|\leq 6$ and $|V(P)|\leq 4$, then $E(P)\cup E(\mathcal{C})$ can be decomposed into $|\mathcal{C}|+1$ paths.
\end{lem}
\begin{proof}
The proof follows by induction on $|\mathcal{C}|$.
It is easy to see that the case when $|\mathcal{C}|=0$ holds.
If $|\mathcal{C}|=1$, let $C$ be the only cycle in $\mathcal{C}$, then $|V(P)\cap V(C)|\leq 4$ since $|V(P)|\leq 4$.
So $E(P)\cup E(C)$ is not isomorphic to the exceptional graph.
Thus by Lemma \ref{lem-K_5^-}, $E(P)\cup E(C)$ can be decomposed into two paths.

Now suppose that $|\mathcal{C}|\geq 2$ and the lemma holds for smaller $|\mathcal{C}|$. Let $P=v_1v_2\cdots v_k$, where $k\leq 4$. We may assume that $v_i$ $(1\leq i\leq k)$ is the vertex of $P$ with smallest index that belongs to $V(\mathcal{C})$, say $v_i\in V(C_1)$ for some $C_1\in \mathcal{C}$. Let $x_1$ and $x_2$ be two neighbors of $v_i$ on $C_1$.
If one of $x_1$ and $x_2$, say $x_1$, is not in $V(P)$,
let $P'=P+x_1v_i-E(v_1Pv_i)$, $P_1=C_1+E(v_1Pv_i)-x_1v_i$, and $\mathcal{C}'=\mathcal{C}\backslash \{C_1\}$. Clearly, $|V(P')|+|\mathcal{C}'|\leq 6$. If $|V(P')|\leq 4$, then by inductive hypothesis, $E(P')\cup E(\mathcal{C}')$ can be decomposed into $|\mathcal{C}'|+1$ paths. Together with $P_1$, we obtain $|\mathcal{C}|+1$ paths, as desired. Thus, assume that either $|V(P')|=5$, or $x_1,x_2\in V(P)$. In both cases, it follows that $k=4$, $i=1$ and $|\mathcal{C}|=2$. Let $\mathcal{C}=\{C_1,C_2\}$.

Next, we show that
$E(P)\cup E(C_1)\cup E(C_2)$ can be decomposed into three paths.
If $|V(P')|=5$, then $P'\cup C_2$ is not isomorphic to the exceptional graph (For otherwise, $C_2$ is a cycle of length 5 with $C_2=v_1v_3x_1v_2v_4v_1$ (see Figure 2($a$)), it follows that $x_2\notin V(P)\cup V(C_2)$, then we can redefine $P'=P+x_2v_1$, $P_1=C_1-x_2v_1$.).
By Lemma \ref{lem-K_5^-}, $E(P')\cup E(C_2)$ can be decomposed into two paths $P_2$ and $P_3$, thus $\{P_1, P_2,P_3\}$ is the desired decomposition.
If $x_1,x_2\in V(P)$,
then $\{x_1,x_2\}=\{v_3,v_4\}$ (see Figure 2 ($b$)). Let $x$ be the other neighbor of $v_3$ on $C_1$. It is easy to check that $x\notin V(P)$. Let $P''=xv_3v_2v_1v_4$, $P_4=C_1-xv_3-v_1v_4+v_3v_4$. Then $P''\cup C_2$ is not isomorphic to the exceptional graph, for otherwise, $v_3v_4\in E(C_2)$, which is a contradiction to the fact that $v_3v_4\in E(P)$. Thus, by Lemma \ref{lem-K_5^-}, $E(P'')\cup E(C_2)$ can be decomposed into two paths $P_5$, $P_6$, and then $\{P_4, P_5, P_6\}$ is as desired.
\end{proof}

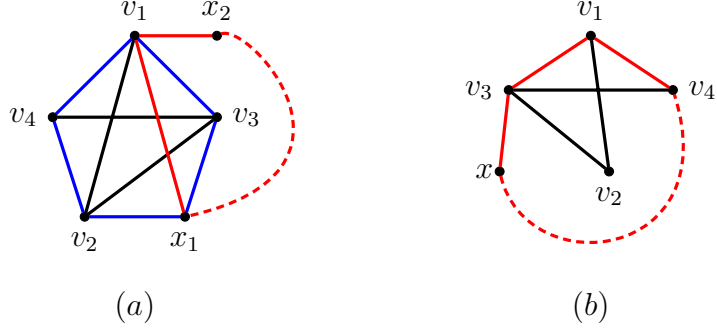
\begin{figure}[htbp]
\begin{center}
\begin{tikzpicture}[scale=0.6]

 \node  at(0,0) [circle, draw=black, fill=black, scale=0.3pt] (v1) {};
 \node [above] at (v1.north) {$v_1$};
 \node at (-1.1,-4)[circle, draw=black, fill=black, scale=0.3pt] (v2)  {};
 \node [below] at (v2.south) {$v_2$};
 \node  at (1.8,-1.8)  [circle, draw=black, fill=black, scale=0.3pt] (v3)  {};
 \node [right] at (v3.east) {$v_3$};
 \node at (-1.8,-1.8)   [circle, draw=black, fill=black, scale=0.3pt] (v4) {};
 \node [left] at (v4.west) {$v_4$};
 \node at (1.1,-4)   [circle, draw=black, fill=black, scale=0.3pt]  (x1) {};
 \node [below] at (x1.south) {$x_1$};
 \node at (1.8,0)   [circle, draw=black, fill=black, scale=0.3pt]  (x2) {};
 \node [above] at (x2.north) {$x_2$};

 \draw [very thick](v1)--(v2)--(v3)--(v4);
 \draw [very thick,blue](v1)--(v4)--(v2)--(x1)--(v3)--(v1);
 \draw [very thick, red](v1)--(x1);
 \draw [very thick, red](v1)--(x2);
 \draw [very thick,densely dashed,red](x2) .. controls (2.5,0.3) and (5.5,-3) .. (x1);
 \node at (0,-6)  {$(a)$};

 \node  at(10,0) [circle, draw=black, fill=black, scale=0.3pt] (v1) {};
 \node [above] at (v1.north) {$v_1$};
 \node at (8.2,-1.2)   [circle, draw=black, fill=black, scale=0.3pt] (v3) {};
 \node [left] at (v3.west) {$v_3$};
 \node  at (11.8,-1.2)  [circle, draw=black, fill=black, scale=0.3pt] (v4)  {};
 \node [right] at (v4.east) {$v_4$};
 \node at (8,-3)[circle, draw=black, fill=black, scale=0.3pt] (x)  {};
 \node [left] at (x.east) {$x$};
 \node at (10.4,-3)   [circle, draw=black, fill=black, scale=0.3pt]  (v2) {};
 \node [below] at (v2.south) {$v_2$};

 \draw [very thick](v1)--(v2)--(v3)--(v4);
 \draw [very thick,red](v1)--(v3);
 \draw [very thick, red](v1)--(v4);
 \draw [very thick,red](x)--(v3);
 \draw [very thick, densely dashed,red](x) .. controls (8.6,-5.5) and (13,-5).. (v4);
 \node at (10,-6)  {$(b)$};

\end{tikzpicture}
\end{center}\caption{ $C_1$, $C_2$ and $P$ are in red, blue and black, respectively.} \label{fig2}
\end{figure}

The following lemma finds a path decomposition when all the cycles have a common vertex.

\begin{lem}\label{6cycle-7paths}
Let $\mathcal{C}$ be a set of edge-disjoint cycles such that all elements in $\mathcal{C}$ contain a common vertex. If $|\mathcal{C}|\leq 6$, then $E(\mathcal{C})$ can be decomposed into $|\mathcal{C}|+1$ paths.
\end{lem}
\begin{proof}
If $|\mathcal{C}|\leq 5$, we apply Lemma \ref{lem-at most 5 cycles to paths} to $\mathcal{C}$ and $P=v$ ($v$ is a common vertex on every cycle in $\mathcal{C}$) and conclude that $E(\mathcal{C})$ can be decomposed into $|\mathcal{C}|+1$ paths.
So we only need to consider the case when $|\mathcal{C}|=6$, and aim to show that $E(\mathcal{C})$ can be decomposed into seven paths. Let $\mathcal{C}=\{C_1,\ldots, C_6\}$.

We first claim that $V(C_i)=V(C_j)$ for every $i\neq j$. For otherwise, let $u$ and $v$ be two vertices such that $v$ is contained in every cycle in $\mathcal{C}$ and $u$ is a neighbor of $v$ on $C_i$ that is not on $C_j$. Let $P_1=C_j-zv+vu$ where $zv\in E(C_j)$, $P_2=C_i-vu$, $\mathcal{C}'=\mathcal{C}\backslash \{C_i, C_j\}$ and $P'=zv$. Then $|V(P')|+|\mathcal{C}'|=2+4=6$. By Lemma \ref{lem-at most 5 cycles to paths}, $E(P')\cup E(\mathcal{C}')$ can be decomposed into five paths, which together with $P_1$ and $P_2$, is a desired path decomposition.

Pick $v_1v_2\in E(C_1)$, $v_3\in N_{C_2}(v_2)$ and $v_4\in N_{C_3}(v_3)\backslash \{v_1\}$. Clearly, $v_1v_2v_3v_4$ is a path. Let $v_5\in N_{C_4\cup C_5 \cup C_6}(v_4)\backslash\{v_1, v_2\}$. Then $v_1v_2v_3v_4v_5$ is also a path. Without loss of generality, assume that $v_4v_5\in E(C_4)$. Choose $v_6\in N_{C_5\cup C_6}(v_5)\backslash\{v_1, v_2, v_3\}$. It is clear that $v_1v_2v_3v_4v_5v_6$ is a path. Similarly, we may assume that $v_5v_6\in E(C_5)$. Let $x_1$, $x_2$ be two neighbors of $v_6$ on $C_6$ and $y_1$, $y_2$ be two neighbors of $v_1$ on $C_6$, respectively.
If $x_1\notin \{v_1, v_2, v_3, v_4\}$, set $P''=v_1v_2v_3v_4v_5v_6x_1$, $P_6'=C_6-v_6x_1$ and $P_i'=C_i-v_{i}v_{i+1}$ for $i=1,\ldots,5$. Then $\{P'', P_1',\ldots, P_6'\}$ is a desired path decomposition of $E(\mathcal{C})$.
So by symmetry, we may assume that $\{x_1,x_2,y_1,y_2\}\subseteq\{v_1, v_2, v_3,v_4,v_5,v_6\}$.

If $v_6\in \{y_1,y_2\}$, that is $v_1v_6\in E(C_6)$, then without loss of generality, assume that $y_1=v_6$ and $x_1=v_1$.
Let $C=v_1v_2v_3v_4v_5v_6v_1$.  For $i=2, \ldots, 5$, let $v_i'\in N_{C_{i-1}}(v_i)\backslash\{v_{i-1}\}$ and $v_i''\in N_{C_i}(v_i)\backslash\{v_{i+1}\}$. Then there exists some $i$ such that $v_i'\notin V(C)$ or $v_i''\notin V(C)$. For otherwise, it is easy to check that for all index $i\in \{2,3,4,5\}$, $v_iv_i', v_iv_i''$ are distinct chords of $C$ and $y_2v_1, x_2v_6\in E(C_6)$ are two chords of $C$ which are different from $v_iv_i'$ and $v_iv_i''$. This contradicts to the fact that a cycle of length 6 has at most nine chords.
If $v_i'\notin V(C)$, let $P''=C-v_{i-1}v_i+v_iv_i'$, $P_{i-1}'=C_{i-1}-v_{i}v_i'$, $P_6'=C_6-v_1v_6$, and $P_j'=C_j-v_{j}v_{j+1}$ for $j\in \{1,2,3,4,5\}\backslash \{i-1\}$; If $v_i''\notin V(C)$, let $P''=C-v_iv_{i+1}+v_iv_i''$, $P_{i}'=C_{i}-v_iv_i''$, $P_6'=C_6-v_1v_6$, and $P_j'=C_j-v_{j}v_{j+1}$ for $j\in \{1,2,3,4,5\}\backslash \{i\}$. In both cases, $\{P'', P_1',\ldots, P_6'\}$ is a path decomposition of $E(\mathcal{C})$, as desired.
Thus, we may assume that $v_6 \not\in \{y_1,y_2\}$. By symmetry, we have $\{y_1,y_2\}\subseteq \{v_3,v_4,v_5\}$ and $\{x_1,x_2\}\subseteq \{v_2,v_3,v_4\}$.

Now without loss of generality, suppose that $v_i=y_1$ and $v_j=y_2$ with $3\leq i<j\leq 5$. Let $w\in N_{C_{i-1}\cup C_{j-1}}(v_6)\backslash \{v_1,v_2,v_3,v_4\}$
(such vertex $w$ must exist since $\{x_1,x_2\}\subseteq \{v_2,v_3,v_4\}$).
For $s\in \{i,j\}$, set $P''=wv_6\ldots v_sv_1v_2\ldots v_{s-1}$, $P_{s-1}'=C_{s-1}-wv_6$, $P_6'=C_6-v_1v_s$, and $P_t'=C_t-v_tv_{t+1}$ for $t\in \{1,2,3,4,5\}\backslash \{s-1\}$. Then $\{P'', P_1',\ldots, P_6'\}$ is a path decomposition of $E(\mathcal{C})$, as desired.
\end{proof}

The following lemma is obtained by grouping cycles in families of six cycles.

\begin{lem}\label{lem-cycles-6}
Let $G$ be a connected graph consisting of $q$ edge-disjoint cycles such that all cycles contain a common vertex. 

(1) If $q\equiv 0 \pmod 6$, then $G$ can be decomposed into $\frac{7q}{6}$ paths;

(2) If $q\equiv \delta \pmod 6$ for $\delta=1,2,3,4,5$, then $G$ can be decomposed into $\frac{7q}{6}+1-\frac{\delta}{6}$ paths.
\end{lem}

\begin{proof}
We put $6$ cycles into a group.
By applying Lemma \ref{6cycle-7paths} to every group of $6$ cycles, we can decompose $G$ into at most $\frac{7q}{6}$ paths if $q\equiv 0 \pmod 6$.
If $q\equiv \delta \pmod 6$ for $\delta = 1,2,3,4,5$, then the remaining $\delta$ cycles can be decomposed into $\delta+1$ paths by applying Lemma \ref{6cycle-7paths} one more time.
Thus, $G$ has a path decomposition with at most $\frac{7(q-\delta)}{6}+\delta+1=\frac{7q}{6}+1-\frac{\delta}{6}$ paths for $\delta=1,2,3,4,5$.
\end{proof}

Let $\alpha(G)$ and $\beta(G)$ denote the number of vertices of odd degree and the number of non-isolated vertices of even degree in $G$, respectively.
We also need the following result in \cite{CFZ}.

\begin{theorem}[Theorem 3.1 in \cite{CFZ}]\label{CFZ}
If $G$ is a triangle-free graph, then $p(G)\leq \frac{\alpha(G)}{2}+\lfloor \frac{3\beta(G)}{5}\rfloor$.
\end{theorem}

\section{Proof of main theorems}

\noindent
In this section, we prove Theorems \ref{main-thm} and \ref{second-thm}.
First we show Theorem \ref{second-thm}.


\begin{proof}[Proof of Theorem \ref{second-thm}]
Let $R=\{v_{1}, v_{2},\ldots, v_{k}\}$ be a triangle removal set of $G$. Let $G'=G-R$. Since all vertices in $V(G)\backslash R$ have even degree in $G$ and $(N_G(v_i)\cup \{v_i\})\cap (N_G(v_j)\cup \{v_j\})=\emptyset$ for any $i, j\in [k]$ by definition of $R$, all the neighbors of $v_{1}, v_{2},\ldots, v_{k}$ have odd degree in $G'$, and thus $\alpha(G')=d_1+\ldots+d_k$, $\beta(G')=n-k-d_1-\ldots-d_k$ where $d_i=d_G(v_i)$.
Since $G'$ is a triangle-free graph, by Theorem \ref{CFZ},  $G'$ has a path decomposition $\mathcal{P}'$ such that
$$
\begin{aligned}
|\mathcal{P}'|
&\leq \frac{1}{2} (d_1+\ldots+d_k)+\left\lfloor\frac{3}{5}(n-k-d_1-\ldots-d_k)\right\rfloor\\
&\leq \frac{1}{2} (d_1+\ldots+d_k)+\frac{3}{5}(n-k-d_1-\ldots-d_k)\\
&=\frac{3n}{5}+(\frac{d_1}{2}-\frac{3}{5}(1+d_1))
+\ldots
+(\frac{d_k}{2}-\frac{3}{5}(1+d_k)).
\end{aligned}
$$

First consider $v_{1}$. Let $N_G(v_{1})=\{v_{11}, v_{12}, \ldots, v_{1d_1}\}$. For $j=1,2,\ldots, d_1$, we may assume that $P_j$ is a path in $\mathcal{P}'$ containing $v_{1j}$ as one of end vertices since $v_{1j}$ has odd degree in $G'$. (Note that it is possible that $P_j=P_t$ for some $t \neq j\in \{1,2,\ldots, d_1\}$.)
For $j=1,2,\ldots, d_1$, set $P_j':=P_j+v_{1}v_{1j}+v_{1}v_{1t}$ if $P_j=P_t$ with $j\neq t$; or otherwise, $P_j':=P_j+v_{1}v_{1j}$.
Note that $P_j'$ is a cycle in the first case, or a path in the second case.
If every $P_j'$ is a path, let $\mathcal{P}=(\mathcal{P}'\backslash\{P_1,\ldots,P_{d_1}\})\cup \{P_1',\ldots,P_{d_1}'\}$, then $\mathcal{P}$ is a path decomposition of $G'+v_1$ with $|\mathcal{P}|=|\mathcal{P}'|$. Otherwise, let $\{P_1',\ldots,P_q'\}$ be the set of all cycles that appear after the above operation. Then $1\leq q\leq \frac{d_1}{2}$.
By Lemma \ref{lem-cycles-6}, $E(P_1')\cup E(P_2')\cup \ldots \cup E(P_q')$ can be decomposed into at most $t$ paths $P_1'',\ldots,P_t''$, where $t=\frac{7q}{6}$ if $q\equiv 0 \pmod 6$, or $t=\frac{7q}{6}+1-\frac{\delta}{6}$ if $q\equiv \delta \pmod 6$ for $\delta=1,2,3,4,5$.
Let $\mathcal{P}=(\mathcal{P}'\backslash\{P_1,\ldots,P_{d_1}\})\cup \{P_1'',\ldots,P_t''\}\cup \{P_j': P_j'  \text{ is a path}, j\in\{1,\ldots,d_1\}\}$. Then $\mathcal{P}$ is a path decomposition of $G_1=G'+v_{1}=G-\{v_{2},\ldots, v_{k}\}$ with
$$|\mathcal{P}|=\left\{
\begin{aligned}
&|\mathcal{P}'|-q+\frac{7q}{6},
~\text{if}~q\equiv 0 \pmod 6;\\ &|\mathcal{P}'|-q+\frac{7q}{6}+1-\frac{\delta}{6},
~\text{if}~q\equiv \delta \pmod 6~\text{for}~
\delta=1,2,3,4,5.
\end{aligned}
\right.
$$

Since $q \le \frac{d_1}{2}$, we have

$$|\mathcal{P}|\leq \frac{3n}{5}
+(\frac{d_2}{2}-\frac{3}{5}(1+d_2))+\ldots
+(\frac{d_k}{2}-\frac{3}{5}(1+d_k))$$
unless $q \equiv 1,2 \pmod 6$.

When $q \equiv 2 \pmod 6$,  since $d_1 \ge 4$, we have
$$\begin{aligned}
|\mathcal{P}|
&\leq \frac{3n}{5}
+\frac{d_1}{12} + 1 - \frac{2}{6}
+(\frac{d_1}{2}-\frac{3}{5}(1+d_1))
+(\frac{d_2}{2}-\frac{3}{5}(1+d_2))+\ldots
+(\frac{d_k}{2}-\frac{3}{5}(1+d_k)) \\
&\leq \frac{3n}{5}
-\frac{d_1}{60}+\frac{1}{15}
+(\frac{d_2}{2}-\frac{3}{5}(1+d_2))+\ldots
+(\frac{d_k}{2}-\frac{3}{5}(1+d_k)) \\
&\leq \frac{3n}{5}
+(\frac{d_2}{2}-\frac{3}{5}(1+d_2))+\ldots
+(\frac{d_k}{2}-\frac{3}{5}(1+d_k)). \\
\end{aligned}
$$

When $q \equiv 1 \pmod 6$, we have
$$\begin{aligned}
|\mathcal{P}|
&\leq \frac{3n}{5}
+\frac{q}{6} + 1 - \frac{1}{6}
+(\frac{d_1}{2}-\frac{3}{5}(1+d_1))
+(\frac{d_2}{2}-\frac{3}{5}(1+d_2))+\ldots
+(\frac{d_k}{2}-\frac{3}{5}(1+d_k)) \\
&\leq \frac{3n}{5}
+\frac{q}{6}-\frac{d_1}{10}+\frac{7}{30}
+(\frac{d_2}{2}-\frac{3}{5}(1+d_2))+\ldots
+(\frac{d_k}{2}-\frac{3}{5}(1+d_k)). \\
\end{aligned}
$$

Since $q \le \frac{d_1}{2}$, we have that $\frac{q}{6}-\frac{d_1}{10}+\frac{7}{30} \le \frac{d_1}{12}-\frac{d_1}{10}+\frac{7}{30} = \frac{14-d_1}{60} \le 0$ if $d_1 \ge 14$.
If $d_1 \le 13$, then $q \le 6$.
Since $q \equiv 1 \pmod 6$, $q = 1$.
Thus, $\frac{q}{6}-\frac{d_1}{10}+\frac{7}{30}=\frac{1}{6}-\frac{d_1}{10}+\frac{7}{30} \le 0$ since $d_1 \ge 4$.
Therefore, we have
$$\begin{aligned}
|\mathcal{P}|
&\leq \frac{3n}{5}
+(\frac{d_2}{2}-\frac{3}{5}(1+d_2))+\ldots
+(\frac{d_k}{2}-\frac{3}{5}(1+d_k)). \\
\end{aligned}
$$

Performing the same operations on every vertex in $\{v_{2}, v_{3},\ldots, v_{k}\}$, we obtain that $G$ has a path decomposition with at most $\frac{3n}{5}$ paths.
This completes the proof.
\end{proof}

Now Theorem \ref{main-thm} is a direct consequence of Theorem \ref{second-thm}.

\begin{proof}[Proof of Theorem \ref{main-thm}]
Let $k$ denote the number of triangles in $G$.
If $k=0$, it follows from Theorem \ref{CFZ}.
Now suppose $k\geq 1$.
Let $\mathcal{T}=\{T_1, T_2, \ldots, T_k\}$ be the set of triangles in $G$, and for each $i=1,2,\ldots, k$, let $v_{i}$ denote a vertex of $T_i$ with maximum $d_G(v_i)$.
Since $G$ is Eulerian on $n \ge 4$ vertices, every triangle $T_i$ has at least one vertex of degree at least $4$. So $d_G(v_i)\geq 4$. Since the distance between any two triangles is at least 3, $(N_G(v_i)\cup \{v_i\})\cap (N_G(v_j)\cup \{v_j\})=\emptyset$ for any $i, j\in [k]$. So $\{v_1,\ldots,v_k\}$ is a triangle removal set of $G$.
It follows from Theorem \ref{second-thm} that $p(G)\leq\frac{3n}{5}$.
This completes the proof.
\end{proof}

\end{document}